\documentclass[12pt,reqno]{article}
\UseRawInputEncoding 

\usepackage[top=2.5cm,bottom=2.5cm,left=2.5cm,right=2.5cm]{geometry}
\usepackage{graphicx}
\usepackage{latexsym}
\usepackage{stmaryrd}
\usepackage{amsmath,amssymb} 
\usepackage{amscd}
\usepackage[arrow,matrix]{xy}
\usepackage{graphicx}
\usepackage{xcolor}
\usepackage{times}
\usepackage[]{subfigure}
 \usepackage{cite}
 \usepackage{float}

\usepackage{hyperref}

\usepackage[colorinlistoftodos]{todonotes}
\newcommand\on{\operatorname}
\renewcommand\div{\on{div}}

\newcommand\Ric{\on{Ric}}

\usepackage{amscd,amsmath,amsthm,amssymb}

\theoremstyle{plain}
\newtheorem{theorem}{Theorem}[section]
\newtheorem{proposition}[theorem]{Proposition}
\newtheorem{corollary}[theorem]{Corollary}
\newtheorem{lemma}[theorem]{Lemma}

\theoremstyle{definition}

\newtheorem{definition}[theorem]{Definition}

\begin{document}

\title{Integral formulas  in  $h$-Almost Ricci-Bourguignon Solitons}
\author {{Abdou Bousso$^{1}$}\thanks{{
 E--mail: \texttt{abdoukskbousso@gmail.com} (A. Bousso)}},\texttt{ }Moctar Traore$^{2}$\footnote{{
 E--mail: \texttt{moctar.traore@ogr.iu.edu.tr} (M. Traore)}}, \texttt{ }Ameth Ndiaye$^{3}$\footnote{{
 E--mail: \texttt{ameth1.ndiaye@istanbul.edu.sn} (A. Ndiaye)}}, \\
 \begin{small}{D\'epartement de Math\'ematiques et Informatique, FST, Universit\'e Cheikh Anta Diop,  \.Dakar, S\'en\'egal.}\end{small}\\ 
\begin{small}{$^{2}$\.{I}stanbul University, Faculty of Science, Department of Mathematics, Vezneciler, 34134, \.{I}stanbul, Turkey.}\end{small}\\
\begin{small}{$^{3}$D\'epartement de Math\'ematiques, FASTEF, Universit\'e Cheikh Anta Diop,  \.Dakar, S\'en\'egal.}\end{small}}
\date{}
\maketitle%


\begin{abstract} 
 The aim of this paper is to investigate some integral formulas for compact gradient $h$-almost Ricci-Bourguignon solitons. Consequently, we  generalize the results previously obtained for Ricci almost solitons.  Moreover, we prove that a compact, non-trivial $h$-almost Ricci-Bourguignon soliton with dimension greater than or equal to $3$ is isometric to a Euclidean sphere, provided either the potential vector field is conformal or its scalar curvature is constant. Finally, we generalize the integral formula for compact $h$-almost Ricci-Bourguignon.  
\end{abstract}
\begin{small} {\textbf{MSC:} 53C25, 53C21, 53E20.}
\end{small}\\
\begin{small} {\textbf{Keywords:} $h$-almost Ricci-Bourguignon solitons, conformal vector fields, integral formulas,  Killing vector fields.} 
\end{small}\\
\maketitle

\section{Introduction}
The concept  of Ricci-Bourguignon flow was initially  introduced by Jean  Pierre Bourguignon in \cite{bib10}. The Ricci-Bourguignon flow  is defined as  an extension of Ricci flow \cite{bib18}.   Richard S. Hamilton \cite{bib17} defined Ricci solitons as a self-similar solutions of Ricci flow. A plethora of examples were presented to elucidate the generalizations and particularizations of gradient 
Ricci-Bourguignon solitons. Then Catino-Mazzieri \cite{bib14} investigated a significant result concerning Einstein manifolds. The objective was to provide the well-known results for rotationally symmetric and asymptotically cylindrical Einstein manifolds in compact or noncompact conditions.     %
 Therefore,  Thomas Ivey \cite{bib19}  investigated a complete classification for Ricci solitons on compact three-manifolds. In a gradient shrinking soliton, Ni and  Wallach  \cite{bib22} studied a complete manifold which provides an alternate proof to a result of Perelman \cite{bib23}. Perelman's work has demonstrated that a compact Ricci soliton is invariably gradient. Indeed,  the existence of gradient Ricci solitons is not guaranteed in the noncompact case  in the non compact case  \cite{bib20}. 
Moreover   for a compact gradient almost Ricci-Bourguignon soliton, some  integral formulas were derived in \cite{bib16, Chen, bib24}.  
As a consequence of integral formula, the author investigated that there exists an isometry between Euclidean sphere and compact gradient soliton  when the scalar curvature is constant  or  the related potential vector field of the manifold is conformal. Also  a new type of Ricci solitons namely Ricci-Bourguignon $h$-almost solitons were studied in  \cite{Gom}. \\

Besides, several authors studied the almost Ricci-Bourguigon solitons (\cite{Mk, Bic, bib30, Filo}).    
Moreover, generalization results of  Ricci solitons were given in  \cite{bib16}. In this study, inspiring the work of Ricci almost solitons, he initiated  the concept of almost Ricci-Bourguignon solitons. He  gave  some important results which were qualified as the generalizing  results for Ricci almost solitons. 
The almost $\eta$-Ricci-Bourguignon solitons provided some special potential vector fields  on a doubly warped product was studied in (\cite{bib6}). The almost $\eta$-Ricci-Bourguignon solitons on compact and non compact case with some special vector field were investigated by Traore et al. (\cite{bib30}). 
 In addition, Barros and Riberio (\cite{bib8, bib9}) presented an integral formula for the compact almost Ricci solitons and generalized $m$-quasi Einstein metrics. 
 Further, Gomes et al. \cite{Gom}  studied the $h$-almost Ricci soliton, where they showed that there exists an isomety between a standard sphere and compact nontrivial case.  \\

Motivated by the above works, in this paper, we study $h$-almost Ricci-Bourguignon solitons on Riemanian manifolds.
The paper is organized as follow: in Sect. $2$, we start with an impportant lemma which allows us to prove that a compact $h$-almost Ricci-Bourguignon solitons with potential vector field satisfying the Hodge de De-Rham decomposition theorem is Killing vector field. Then with some integral friomulas, we prove that a compact, non-trivial $h$-almost Ricci-Bourguignon solitons with dimension greater than or equal to $3$ is isometric to an Euclidean sphere, provided either is potential vector field is conformal or its  scalar curvature  is constant. At the end, we generalize the integral formula for compact $h$-almost Ricci-Bourguignon.

\section{Preliminaries}

In this part, we recall the fundamental definitions and concepts for the further study. 
 Let  $(M^n, g)$ be an  $n$-dimensional Riemannian manifold, then we defined on $M^n$ the {\it Ricci-Bourguignon solitons} as a  self-similar solutions to {\it Ricci-Bourguignon flow} \cite{bib13*} defined by:
\begin{equation} \label{eq1}
\frac{\partial }{\partial t}g(t)=-2(\Ric-\rho R g), 
\end{equation}
where $R$ is the scalar curvature of the Riemannian metric $g$, $\Ric$ is the Ricci curvature tensor of the metric,  and $\rho$ is a real constant. 
 When  $\rho=0$ in ~\eqref{eq1}, then  we get a Ricci flow. 
Remark that for some  values of $\rho$ in equation~\eqref{eq1}, $\Ric-\rho Rg$ evolve into  the following situations:  \\

\textbf (1)  If  $\rho=\frac{1}{2},$ then it is an Einstein tensor $\displaystyle\Ric-\frac{R}{2}g$, 

\textbf (2)   If  $\rho=\frac{1}{n},$ then it is a traceless Ricci tensor $\displaystyle\Ric-\frac{R}{n}g$.

\begin{definition}\cite{bib2}
Let  $(M^n, g)$  be a  Riemannian manifold of dimension $n \geq 3$. Then it  is called  Ricci-Bourguignon soliton if
\begin{equation} \label{eq2}
\Ric+\frac{1}{2}\mathcal{L}_{\xi}g=(\lambda+\rho R)g,
\end{equation}
where  $\mathcal{L}_{\xi}$ denotes the Lie derivative operator along the vector field $\xi$ which is called soliton or potential, $\rho$  and  $\lambda$ are real constants and the soliton structure is  denoted by  $(M, g, \xi, \lambda)$. If $\rho$  and  $\lambda$ are smooth functions, then it  is called almost Ricci-Bourguignon soliton
\end{definition}

If  $\lambda<0,$ $\lambda=0$ or  $\lambda>0$,  then the Ricci-Bourguignon soliton  is called   expanding, steady or  shrinking  respectively.

\begin{definition} \cite{Haj}
Let $(M^n, g)$ be a Riemannian manifold. Then it is called  $h$-almost  Ricci-Bourguignon soliton  if there exist a vector field $\xi$,  the smooth functions $\lambda$ and $h$  such that 
\begin{equation} \label{eq3h}
\Ric+\frac{h}{2}\mathcal{L}_{\xi}g=(\lambda+\rho R)g,
\end{equation}
where $\pounds_{\xi}g$ denotes the Lie derivative of the metric $g$ along $\xi$. 
Hence it is  denoted by $(M^n, g, \xi, \lambda, \rho)$.
\end{definition}
Therefore, if $\xi=\nabla f$, we get a gradient $h$-almost Ricci-Bourguignon soliton.  Hence,  equation \eqref{eq3h} is rewritten as
\begin{equation} \label{eq4h}
\Ric+h\nabla^2 f=(\lambda+\rho R)g,
\end{equation}
here $\nabla^2 f$ is the Hessian of $f$.
Thus for $\rho=\frac{1}{2}$, we have a  gradient $h$-almost  Einstein soliton  and for $\rho=\frac{1}{n}$, we get a gradient  $h$-almost traceless Ricci soliton. \\

 If $f$ is constant, then the gradient $h$-almost Ricci-Bourguignon soliton is called  trivial. On the other hand, for  a nontrivial gradient $h$-almost Ricci-Bourguignon soliton $f$  is not trivial. Hence  a gradient $h$-almost Ricci-Bourguignon soliton is  an Einstein manifold when $n\geq3$ and $f$ is constant.


On a Riemannian manifold the following well-known lemma holds  \cite{bib7}
\begin{lemma} \label{lemi1}
Let  $S$ is a $(0, 2)$-tensor on a Riemannian manifold  $(M^n, g)$. Then we have 
\begin{equation}
\div(S(\phi \xi))=\phi(\div S)(\xi)+\phi \langle\nabla \xi, S\rangle+S(\nabla \phi, \xi),
\end{equation}
for all vector field $\xi$ and  $\phi$ a smooth function on $M$.   
\end{lemma}

\section{Main Results}
Let $f:M^n\rightarrow\mathbb{R}$ be a real differential function on a Riemannian manifold $M^n$. Then the trace-free part of the Hessian of $f$, is given by 
\begin{equation}\label{ekt1}
     \vert\dot{\nabla}^2f\vert^2=\vert\nabla^2f\vert^2-\frac{(\Delta f)^2}{n}. 
 \end{equation}
 Proceeding, we have the following lemma:
\begin{lemma}\label{vc}
    Let \( (M^n,g,\xi,\lambda, \rho) \) be an \(h\)-almost Ricci-Bourguignon soliton. Then the following conditons hold: 
    \[\operatorname{\div}\left((h\mathcal{L}_{\xi}g)\zeta\right)=2\operatorname{\div}(\lambda \zeta)+2\rho\operatorname{\div}(R \zeta)-(1-2\rho)g(\nabla R,\zeta)-2\langle \nabla \zeta,\operatorname{Ric}\rangle.\]
If \(M^n\) is compact then $\displaystyle\int_{M}g(\nabla R,\zeta)\mathrm{d}M=(2\rho-1)\int_{M}\langle \nabla \zeta,\operatorname{Ric}\rangle\mathrm{d}M$,\,  for any  vector field \(\zeta\) on $M$.
    \end{lemma}
    \begin{proof}
 First, taking  $S =\mathcal{L}_{\xi}g$ and $\phi = 1$ in Lemma \eqref{lemi1}. Then using equation  \eqref{eq3h} and the second contracted Bianchi identity $2\div\Ric=\nabla R$, we obtain
\begin{align*}
\displaystyle\operatorname{div}\left((h\mathcal{L}_{\xi}g)\zeta\right)&=\operatorname{div}(h\mathcal{L}_{\xi}g)(\zeta)+\langle \nabla \zeta,h\mathcal{L}_{\xi}g\rangle\\
&=\displaystyle\bigg\{\operatorname{div}\bigg(2\lambda I+\rho R I-2\operatorname{Ric}\bigg)\bigg\}(\zeta)+\bigg\langle \nabla \zeta, -2\operatorname{Ric}+2(\lambda+\rho R)g\bigg\rangle \\
&=\displaystyle-g(\nabla R,\zeta)+2g(\nabla \lambda,\zeta)+2\rho g(\nabla R,\zeta)+2\div\lambda+ 2\rho\div R-2\langle\nabla \zeta,\operatorname{Ric}\rangle\\
&=(2\rho-1)g(\nabla R,\zeta)+2g(\nabla\lambda,\zeta)+2\div\lambda+ 2\div\rho R-2\langle \nabla \zeta,\operatorname{Ric}\rangle,
\end{align*}
   for any vector field $\zeta\in\mathfrak{X}(M)$.  \\
          
     On the other hand if  \( M^n \) is compact, then applying Green's theorem we arrive at 
     \[\displaystyle\int_{M}g(\nabla R,\zeta)\mathrm{d}M=(2\rho-1)\int_{M}\langle \nabla \zeta,\operatorname{Ric}\rangle\mathrm{d}M.\]
     Hence the proof is completed. 
    \end{proof}
When we have a compact $h$-almost Ricci-Bourguignon soliton  manifold, we can use the Hodge–de Rham decomposition theorem to write
\begin{equation} \label{ab1}
 \xi = \nabla h + Y ,
\end{equation}
where $Y\in\mathfrak{X}(M)$ whereas \( h: M^n \to \mathbb{R} \) is a smooth function on $M^n$ and \( \operatorname{div}(Y) = 0 \). Therefore the equation \eqref{eq3h} turns into 
\begin{equation}\label{equ0}
    \operatorname{Ric}+h\nabla^2h+\frac{h}{2}\mathcal{L}_Yg=(\lambda+\rho R)g.
\end{equation}
Now taking $\displaystyle N_{k}=\frac{h}{2}\mathcal{L}_{\xi}g+h\nabla^2h$ in \eqref{equ0}, then it can be rewritten as follows
\begin{equation}\label{equ00}
    \displaystyle\operatorname{Ric}+N_{k}=(\lambda+\rho R)I.
\end{equation}
 Now we recall that the traceless part of a tensor T is defined by

 \begin{equation}
     \dot T=T-\frac{tr(T)}{n}g. 
 \end{equation}
We then calculate
\begin{align*}
\operatorname{div}(N_{k}(\nabla h))&=(\operatorname{div}N_{k})(\nabla h)+\langle \nabla^2h,N_{k}\rangle\\
&=(\operatorname{div}\lambda I+\rho R I-\operatorname{Ric}))(\nabla h)+\langle \nabla^2h, h\frac{1}{2}\mathcal{L}_{\xi}g+h\nabla^2h\rangle \\
&=(\rho-\frac{1}{2})g(\nabla R,\nabla h)+g(\nabla\lambda,\nabla h)+h\vert\nabla^2h\vert^2+h\langle  \nabla^2h, \frac{1}{2}\mathcal{L}_{\xi}g\rangle.
\end{align*}

Using the fact that $\displaystyle\vert\dot\nabla^2h\vert^2+\frac{(\Delta h)^2}{n}=\vert\nabla^2h\vert^2$, we obtain
\begin{align*}
\operatorname{div}(N_{k}(\nabla h))&=(\rho-\frac{1}{2})g(\nabla R,\nabla h)+g(\nabla\lambda,\nabla h)+h\vert\dot\nabla^2h\vert^2\\
&+h\displaystyle\frac{(\Delta h)^2}{n}+h\langle  \nabla^2h, \frac{1}{2}\mathcal{L}_{\xi}g\rangle\\
&=(\rho-\frac{1}{2})g(\nabla R,\nabla h)+g(\nabla\lambda,\nabla h)+h\vert\dot\nabla^2h\vert^2\\
&+h\displaystyle\frac{(\Delta h)^2}{n}+h\langle  \nabla^2h, \frac{1}{2}\mathcal{L}_{\xi}g\rangle\\
&=(\rho-\frac{1}{2})g(\nabla R,\nabla h)+g(\nabla\lambda,\nabla h)+h\vert\dot\nabla^2h\vert^2\\
&+\displaystyle\frac{(\Delta h)}{n}(n\lambda+n\rho R-R)+h\langle  \nabla^2h, \frac{1}{2}\mathcal{L}_{\xi}g\rangle.
            \end{align*}
Thus, we get
\begin{equation}\label{f1}
\begin{array}{ll}
\operatorname{div}(N_{k}(\nabla h))&=(\rho-\frac{1}{2})g(\nabla R,\nabla h)+g(\nabla\lambda,\nabla h)+h\vert\dot\nabla^2h\vert^2\\
&+\displaystyle\lambda\Delta h+(\rho-\frac{1}{n}) R\Delta h +h\langle  \nabla^2h, \frac{1}{2}\mathcal{L}_{\xi}g\rangle.
\end{array}
 \end{equation}
From which we derive the following result.

\begin{theorem}\label{ll}
  Let   $(M^n,g,\zeta,\lambda, \rho)$ be a compact $h$-almost Ricci-Bourguignon soliton. Then the  following integral formula holds: 
  \begin{align*}
\displaystyle\int_Mh\vert\dot\nabla^2h\vert^2 dM&=\displaystyle\int_M\left(\frac{n\rho-1}{n}\right)g(\nabla R,\nabla h)dM+\int_M g(\nabla h,\nabla\lambda)dM+\int_M\vert\nabla^2h\vert^2 dM.
 \end{align*}
\end{theorem}
\begin{proof}
On integrating equation \eqref{f1}, we get

\begin{equation}\label{f1s}
\begin{array}{ll}
\displaystyle\int_Mh\vert\dot\nabla^2h\vert^2=\displaystyle\int_M\left(\frac{n-2}{2n}\right)g(\nabla R,\nabla h)-\int_M\langle  \nabla^2 h, \frac{h}{2}\mathcal{L}_{\xi}g\rangle.
\end{array}
 \end{equation}
Thus, from \eqref{equ0}, we derive 

\begin{equation}\label{u1}
\langle  \nabla^2h, h\frac{1}{2}\mathcal{L}_{\xi}g\rangle=\lambda\Delta h+\rho R\Delta h-\vert\nabla^2h\vert^2-\langle  \nabla^2h, \Ric\rangle.
\end{equation}
Moreover the second contracted Bianchi identity gives
\begin{equation}\label{f12}
\int_Mg(\nabla h, \nabla R) dM=-2\int_M\langle  \nabla^2h, \Ric\rangle dM.
\end{equation}
Using \eqref{f1s}, \eqref{u1} and \eqref{f12}, we get the desired result.   
\end{proof}
As a direct consequence of this theorem we deduce the following corollary.

\begin{corollary}
 Let $(M^n,g,\xi,\lambda, \rho)$ be a compact $h$-Ricci-Bourguignon soliton with potential vector field $\zeta=\nabla h+Y$. Then  $\xi$ is a Killing vector field. 
\end{corollary}

\begin{proof}
 Notice that $\lambda$ is constant and and $\xi=\nabla h+Y$ with $\div Y=0$. Then from Lemma \ref{ll}, we get
  \begin{align}\label{tr1}
\displaystyle\int_Mh\vert\dot\nabla^2h\vert^2 dM&=\displaystyle\int_M(\frac{n\rho-1}{n})g(\nabla R,\xi)dM+\int_M\vert\nabla^2h\vert^2 dM.
 \end{align}
  But, using that $\xi$ is conformal we are able to use Kazdan–Warner's identity
 \cite{kazj} to deduce that $\displaystyle\int_Mg(\nabla R,\xi)dM=0$. Then we have 
 \[\displaystyle\int_Mh\vert\dot\nabla^2h\vert^2 dM=\int_M\vert\nabla^2h\vert^2 dM.\] 
 Thus $h$ is constant, which allows us to conclude $\xi=Y$ and $\mathcal{L}_{\xi}g=0$. Hence the proof is completed.
\end{proof}

Following the result from \cite{bib1} in a compact Ricci soliton, we state the following integral formulas for \(h\)-almost Ricci-Bourguignon soliton  in a compact Riemannian manifold.

\begin{proposition}\label{p1}
    Let \(\left(M^n,g,\xi,h,\lambda\right)\) be a compact \(h\)-almost Ricci-Bourguignon soliton. Then the following identities are satisfied:
    \begin{enumerate}
        \item \(\displaystyle\int_{M}h\left|\nabla^2h-\frac{\Delta h}{n}g\right|^2\mathrm{d}M^n=\frac{n-2}{2n}\displaystyle\int_{M}g(\nabla R,\nabla h)\mathrm{d}M-\frac{1}{2}\displaystyle\int_{M}h\langle \nabla^2 h,\mathcal{L}_Yg\rangle\mathrm{d}M\)
        \item \(\displaystyle\int_{M}h\left|\nabla^2 h-\frac{\Delta h}{n}g\right|^2\mathrm{d}M^n=\frac{n-2}{2n}\displaystyle\int_{M}g(\nabla R,\nabla f)\mathrm{d}M^2+\frac{1}{4}\displaystyle\int\limits_{M}h|\mathcal{L}_Yg|^2\mathrm{d}M^n\)
        \item \(\displaystyle\int_{M^n}\frac{1}{h}\left|\operatorname{Ric}-\frac{R}{n}g\right|^2\mathrm{d}M=\frac{n-2}{2n}\displaystyle\int_{M}g(\nabla R,\nabla h)\mathrm{d}M=\frac{n-2}{2n}\displaystyle\int_{M}g(\nabla R,\xi)\mathrm{d}M^n.\)
    \end{enumerate}
    \end{proposition}

    \begin{proof}
   From Lemma \ref{lemi1}, we infer
     \begin{align*}
            \operatorname{div}(\operatorname{Ric}(\nabla f))&=\operatorname{div}\operatorname{Ric}(\nabla h)+\langle \nabla^2h,\operatorname{Ric}\rangle.
        \end{align*}
  Then, using the second contracted Bianchi identity $2\div\Ric=\nabla R$ and the equation  $\displaystyle|\nabla^2h|^2=\left|\nabla^2h-\frac{\Delta f}{n}g\right|^2+\frac{(\Delta h)^2}{n}$, we get 
  \begin{align}\label{eg1}
  \begin{array}{ll}
            \displaystyle\operatorname{div}(\operatorname{Ric}(\nabla h))&=
            \displaystyle\frac{1}{2}g(\nabla R,\nabla h)+\langle \nabla^2 h,-h\nabla^2h-\frac{h}{2}\mathcal{L}_Yg+\lambda g+\rho R g \rangle\\
            &=\displaystyle\frac{1}{2}g(\nabla R,\nabla h)-h|\nabla^2h|^2+\lambda \Delta h+\rho R \Delta h-\frac{h}{2}\langle \nabla^2f,\mathcal{L}_Yg\rangle\\
             &=\displaystyle\frac{1}{2}g(\nabla R,\nabla h)+\lambda \Delta f+\rho R \Delta h-h\left|\nabla^2h-\frac{\Delta h}{n}g\right|^2-h\frac{(\Delta f)^2}{n}-\frac{h}{2}\langle \nabla^2h,\mathcal{L}_Yg\rangle\\
             &=\displaystyle\frac{1}{2}g(\nabla R,\nabla h)+\frac{R}{n}\Delta h-h\left|\nabla^2h-\frac{\Delta h}{n}g\right|^2-\frac{h}{2}\langle \nabla^2h,\mathcal{L}_Yg\rangle.
              \end{array}
        \end{align}

 By definition, we have \[\operatorname{div}(R\nabla h)=R\operatorname{div}(\nabla h)+g(\nabla R,\nabla h)=R\Delta h+ g(\nabla R,\nabla h).\]
        Since $M$ is compact, thus applying Green theorem to the previous equality, 
       we deduce that \[\int_{M}R\Delta h \mathrm{d}M=-\int_{M}g(\nabla R, \nabla h)\mathrm{d}M.\]
        On integrating the equation \eqref{eg1}, we get
        \begin{equation}\label{equ3}
            \int_{M}h\left|\nabla^2h-\frac{\Delta h}{n}g\right|^2\mathrm{d}M=\frac{n-2}{2n}\int_{M}g(\nabla R,\nabla h)\mathrm{d}M-\frac{1}{2}\int_{M}h\langle \nabla^2h,\mathcal{L}_Yg\rangle\mathrm{d}M,
        \end{equation}
        which allow us to get the first item. \\
        
        For proving the second item, using equation \eqref{eq4h}, we get 
        
        \[h\langle\nabla^2h,\mathcal{L}_Yg\rangle=\langle h\nabla^2h,\mathcal{L}_Yg\rangle=\left\langle-\operatorname{Ric}-\frac{h}{2}\mathcal{L}_Yg+(\lambda+\rho R)g,\mathcal{L}_Yg\right\rangle.\]
        Taking into account that 
        
        \[\frac{1}{2}\langle \operatorname{Ric},\mathcal{L}_Yg\rangle=\langle \operatorname{Ric},\nabla Y\rangle,\]
        we deduce 
\begin{equation}\label{er1}
\begin{array}{ll}
        \displaystyle\int_{M}h\langle\nabla^2h,\mathcal{L}_Yg\rangle \mathrm{d}M^n=&-\displaystyle\int_{M}\langle \operatorname{Ric},\mathcal{L}_Yg\rangle \mathrm{d}M^n-\frac{1}{2}\int\limits_{M}h|\mathcal{L}_Yg|^2\mathrm{d}M\\
        &+\displaystyle\int_{M}(\lambda+\rho R)\operatorname{div}(Y)\mathrm{d}M.
        \end{array}
    \end{equation}
    Using the fact that
        \[\int_{M}\langle \operatorname{Ric},\mathcal{L}_Yg\rangle \mathrm{d}M=\int_{M}\operatorname{div(Ric(Y))\mathrm{d}M=0},\]
   then equation  \eqref{er1} turns into
       \begin{equation}\label{er2}
  \int\limits_{M}h\langle \nabla^2h,\mathcal{L}_Yg\rangle\mathrm{d}M=-\frac{1}{2}\int_{M}h|\mathcal{L}_Yg|^2\mathrm{d}M.
        \end{equation}
   Now plugging  \eqref{er2} into \eqref{equ3}, we obtain
         \begin{equation}\label{equ4}
            \int\limits_{M}h\left|\nabla^2h-\frac{\Delta h}{n}g\right|^2\mathrm{d}M=\frac{n-2}{2n}\int_{M}g(\nabla R,\nabla h)\mathrm{d}M+\frac{1}{4}\int_{M}|\mathcal{L}_Yg|^2\mathrm{d}M.
        \end{equation}
        Hence the second item follows.\\
        
       \noindent For proving the last item, we consider 
        \begin{align*}
        \displaystyle\left|\operatorname{Ric}-\frac{R}{n}g\right|^2&=\left|-\frac{h}{2}\mathcal{L}_\xi g+(\lambda+\rho R)g-\frac{R}{n}g\right|^2=\left|\frac{h\Delta f}{n}g-\frac{h}{2}\mathcal{L}_\xi g\right|^2\\
        &=h^2\left(\frac{(\Delta f)^2}{n}+\frac{1}{4}|\mathcal{L}_\xi g|^2-\frac{\Delta f}{n}\langle 
       g,\mathcal{L}_\xi g\rangle\right)\\
       &=h^2\left(\frac{(\Delta f)^2}{n}+\frac{1}{4}\left(4|\nabla^2f|^2+|\mathcal{L}_Y g|^2+4\langle \nabla^2f,\mathcal{L}_Yg\rangle\right)-\frac{2(\Delta f)^2}{n}\right)\\
       &=h^2\left|\nabla^2f-\frac{\Delta f}{n}g\right|^2+h^2\left(\frac{1}{4}|\mathcal{L}_Yg|^2+\langle \nabla^2f,\mathcal{L}_Yg\rangle\right).
       \end{align*}
       Therefore multiplying the previous equation by $\displaystyle\frac{1}{h}$, we obtain
       \[\int_{M}\frac{1}{h}\left|\operatorname{Ric}-\frac{R}{n}g\right|^2\mathrm{d}M=\int_{M}h\left|\nabla^2f-\frac{\Delta f}{n}g\right|^2\mathrm{d}M-\frac{1}{4}\int_{M}h|\mathcal{L}_Yg|^2\mathrm{d}M.\]
       Hence equation  \eqref{equ4} allows us to obtain: 
       \begin{equation*}
           \int_{M}\frac{1}{h}\left|\operatorname{Ric}-\frac{R}{n}g\right|^2\mathrm{d}M=\frac{n-2}{2n}\int_{M}g(\nabla R,\nabla f)\mathrm{d}M.
       \end{equation*}
       Since \( \displaystyle\int_{M}g(\nabla R,Y)\mathrm{d}M= \displaystyle\int_{M}\langle \operatorname{Ric},\nabla Y\rangle\mathrm{d}M^n=0\), it follows that
       \begin{align*}
           \int_{M}\frac{1}{h}\left|\operatorname{Ric}-\frac{R}{n}g\right|^2\mathrm{d}M&=\frac{n-2}{2n}\int_{M}g(\nabla R,\nabla f)\mathrm{d}M\\
           &=\frac{n-2}{2n}\int_{M}g(\nabla R,\xi)\mathrm{d}M.
       \end{align*}
       Thus the proof is completed.
        \end{proof}
We give the following theorem, which is a direct application of the previous proposition.

\begin{theorem}\label{t1}
    Let \( (M^n,g,\xi,h,\lambda) \) be a compact \(h\)-almost Ricci-Bourguignon soliton of dimension  $n>3$. Then $M$ is isometric to an Euclidean sphere if and only if one of the following conditions are satisfied:
    \begin{enumerate}
        \item \(\displaystyle\int_{M}h\langle \nabla^2h,\operatorname{Ric}\rangle \mathrm{d}M=0;\)
        \item The vector field \(\nabla h\) is nontrivial conformal. 
    \end{enumerate}
\end{theorem}
\begin{proof}
First, suppose that $\displaystyle\int\limits_{M^n}\langle \nabla^2h,\operatorname{Ric}\rangle\mathrm{d}M^n=0$. Then from Lemma \ref{vc}, we obtain 
    \begin{align*}
    \displaystyle\int_{M}\langle \nabla^2h,\operatorname{Ric}\rangle\mathrm{d}M=-\frac{1}{2}  \displaystyle\int\limits_{M}g(\nabla R, \nabla h)\mathrm{d}M=-\frac{1}{2}  \displaystyle\int_{M}g(\nabla R, \xi) \mathrm{d}M. 
    \end{align*}

    Therefore, we can use assertion 3 of Proposition \ref{p1}to conclude that $(M^n,g)$  is Einstein. Moreover, from the fundamental equation \eqref{eq4h}, we deduce that $\xi$  is a non-trivial conformal vector field. It is possible to employ the results of Theorem 2 in \cite{bib9} to establish that $(S^n,g_0)$ is isometric to a standard sphere $(\mathbb{S}^n, g_0)$.\\

    It is hypothesised that $\nabla h$ is a conformal vector field. Assuming the validity of Assumption 3 of Proposition 1, it can be deduced that 
$\displaystyle\int_{M}g(\nabla R,\nabla h)\mathrm{d}M\geq 0$. Consequently, Assumption 2 of the aforementioned proposition establishes that $Y$ is a Killing vector field and $\displaystyle\int_{M}g(\nabla R,\nabla h)\mathrm{d}M=0$. Finally, employing Assumption 3 of the aforementioned proposition once more allows the conclusion to be drawn that $(M^n, g)$ is Einstein. Following the same line of argumentation as previously outlined, it can be concluded that $(M^n, g)$ is isometric to a standard sphere $(\mathbb{S}^n, g_0)$.

It is finally observed that a standard sphere $(\mathbb{S}^n, g_0)$ is Einstein. Consequently, if \( (M^n,g,\xi,h,\lambda) \), with $n \geq 3$, is an $h$-almost Ricci-Bourguignon  soliton, it can be deduced that $\xi$ is a conformal vector field.
Utilising the hodge de De-Rham decomposition Theorem $\xi=\nabla h+Y$. Then we deduce that $\nabla h$ is a conformal vector field on  $(\mathbb{S}^n, g_0)$, this result consequently establishes item 2. Furthermore, since $\Ric=(n-1)g_0$, it follows that $\displaystyle\int_{M}g(\Ric,\nabla^2h)\mathrm{d}M=(n-1)\int_M\Delta h\mathrm{d}M=0$. Consequently, the proof of is completed.

    \end{proof}
 Theorem \ref{t1} allows us to state the following corollary.
    \begin{corollary}\label{c1}
        The compact \(h\)-almost Ricci-Bourguignon soliton \( \left(\prod\limits_{k=0}^mM_k^{n_k},\bigoplus\limits_{k\in\llbracket 0,m\rrbracket}g_k,(\xi_0,\cdots,\xi_m),\lambda,\rho' \right)\) is isometric to an \( \sum\limits_{k=0}^mn_k \)-dimensional sphere with \(\sum\limits_{k=0}^mn_k\geq 3\) if and only if one of the following conditions are satisfied: 
        \begin{enumerate}
            \item \(\displaystyle\int_{\prod\limits_{k=0}^mM_k^{n_k}}\left\langle \left(\nabla^2 f_0,\cdots,\nabla^2f_m\right),\bigoplus\limits_{k\in\llbracket 0,m\rrbracket}\operatorname{Ric_k}\right\rangle \textbf{d}\prod\limits_{k=0}^mM_k^{n_k}=0\), with for all \(k\in\llbracket 0,m\rrbracket\),
            
            \(\xi_k=\nabla f_k+Y_k,\quad \operatorname{div}(Y_k)=0.\)  
            \item \(\sum\limits_{k=0}^m\nabla f_k\) is conform.
            \end{enumerate} 
    \end{corollary}

        Now we give a generalization result for $h$-almost Ricci-Bourguignon soliton
        \begin{theorem}\label{p2}
            Let \((M^{n_k}_k,g_k,\xi_k,h,\lambda,\rho)_{k\in\llbracket 0,m\rrbracket}\) be a family of \( h \)-almost Ricci-Bourguignon soliton such that for any \( k \in \llbracket 0, m \rrbracket \). Then, \( \left(\prod\limits_{k=0}^m M^{n_k}_k,\bigoplus\limits_{k\in\llbracket 0,m\rrbracket}g_k,(\xi_0,\cdots,\xi_m),h,\lambda,\frac{\rho}{1+m} \right)\) is also \( h\)-almost Ricci-Bourguignon.
        \end{theorem}
    
\begin{proof}
    From equation \eqref{eq3h}, we have 
    \[\operatorname{Ric_k}+\frac{h}{2}\mathcal{L}_{\xi_k}g=(\lambda+\rho_k R)g_k.\]
    Moreover from the previous identity we found
   \[\bigoplus\limits_{k\in\llbracket 0,m\rrbracket}\operatorname{Ric}_k+\frac{h}{2}\bigoplus\limits_{k\in\llbracket 0,m\rrbracket}\mathcal{L}_{\xi_k}g_k=\left(\lambda+\rho R\right)\bigoplus\limits_{k\in\llbracket 0,m\rrbracket}g_k.\]
    Finally setting \( \rho'=\frac{\rho}{1+m} \), we obtain the result.
\end{proof}

Finally, we give the following result

\begin{theorem}
    Let \((M^n, g, \nabla f, \lambda)\) be a gradient \( h \)-almost Ricci-Bourguignon soliton and define the map  

\[
\mathcal{F}(X, Y) = (\nabla_X \Ric)Y - (\nabla_Y Ric)X
\]

for all vector fields \(X\) and \(Y\) on \(M^n\). Then, we have:

\[
\mathcal{F}(X, Y) = (X \wedge Y)(\nabla \lambda + \rho \nabla R) + h \mathrm{R}(X, Y) \nabla f + (Y(h) \nabla^2 f) X - (X(h) \nabla^2 f) Y.
\]
\end{theorem}
\begin{proof}
Let \( X \) and \( Y \) be two vector fields ON $M$. Then we have
\begin{equation}\label{vc0}
  \mathcal{F}(X,Y)=(\nabla_X\operatorname{Ric})Y-(\nabla_Y\operatorname{Ric})X,  
\end{equation}
where
    $\mathcal{F}(X,Y)+\mathcal{F}(Y,X)=0$.
     On the other hands, taking the covariant derivative of \eqref{eq3h}, we obtain 
    \begin{equation}\label{hc}
    \begin{array}{ll}
        (\nabla_X\operatorname{Ric})Y&=\left(\nabla_X(-h\nabla^2f+(\lambda+\rho R)g)\right)Y \\
        &=\nabla_X(-h\nabla^2f)Y+X(\lambda)Y+\rho X(R)Y\\
        &=-(X(h)\nabla^2f)Y-h(\nabla_Y\nabla^2f)Y+X(\lambda)Y+\rho X(R)Y\\
        &=-(X(h)\nabla^2f)Y-h\left(\nabla_X\nabla_Y\nabla f-\nabla_{\nabla_{X}Y}\nabla f\right)+X(\lambda)Y+\rho X(R)Y.
          \end{array}
    \end{equation}
    Thus plugging \eqref{hc} in \eqref{vc0}, we get
    \begin{align*}
    \mathcal{F}(X,Y)&=(\eta(h)\nabla^2f)\zeta-(X(h)\nabla^2f)\eta\\
    &+h\left(\nabla_Y\nabla_X\nabla f
    -\nabla_X\nabla_Y\nabla f+\nabla_{\nabla_XY}\nabla f-\nabla_{\nabla_YX}\nabla f\right)\\
    & X(\lambda)Y-Y(\lambda)X+\rho\left(X(R)Y-Y(R)X\right)
    \end{align*}
    Thus after a  straightforward we conclude that
    \begin{align*}
        \mathcal{F}(X,Y)&=(\eta(h)\nabla^2f)X-(X(h)\nabla^2f)Y\\
        &+h\left(\nabla_Y\nabla_X\nabla f-\nabla_X\nabla_Y\nabla f+\nabla_{[X,Y]}\nabla f\right)\\
        &+(X\wedge Y)(\nabla\lambda+\rho \nabla R)\\
   &=(X\wedge Y)(\nabla \lambda+\rho \nabla R)\\
   &+h\mathrm{R}(X,Y)\nabla f+(Y(h)\nabla^2f)X-(\zeta(h)\nabla^2f)\eta
    \end{align*}
    Hence the proof is completed.
\end{proof}


\begin{thebibliography}{99}

\bibitem{bib1}  Aquino, C.,  Barros, A., Ribeiro, E. Jr., {\it Some applications of the Hodge-de Rham decomposition to Ricci solitons}, Results  Math., 60(1), 235-246, (2011). 

\bibitem{bib2} Aubin, T.,   {\it Metriques Riemanniennes et courbure}, J. Differ. Geom., 4(4), 383–424, (1970). 
  


\bibitem{bib6} Blaga, A.M., Tastan, H.M.,  {\it Some results on almost $\eta$-Ricci-Bourguignon solitons}, J. Geom. Phys., 168, 104316, (2022).    

\bibitem{bib7} Barros, A., Gomes, J.N., {\it A compact gradient generalized quasi-Einstein metric with constant scalar curvature},  J. Math. Anal. Appl., 401(2), 702–705, (2013). 

\bibitem{bib8}  Barros, A., Ribeiro, E. Jr.,  {\it Characterizations and Integral formulae for generalized $m$-quasi-Einstein metrics},  Bull Braz. Math. Soc, New Series.,  45(2), 325-341, (2014).


\bibitem{bib9} Barros, A., Ribeiro, E.Jr.,  {\it Some characterizations for compact almost Ricci solitons},  Proc. Amer. Math. Soc., 140(3), 1033-1040, (2012). 


\bibitem{bib10}  Bourguignon, J.P.,  {\em Ricci curvature and Einstein metrics}, Global differential geometry and global analysis.,  42–63, 1981.


\bibitem{bib13*} Catino, G., Cremaschi, L., Djadli, Z.,  Mantegazza, C.,  Mazzieri, L.,  {\it The Ricci-Bourguignon flow},  Pac. J. Math., 287(2),  337–370, (2017). 

\bibitem{bib14} Catino, G., L. Mazzieri, L., {\it Gradient Einstein solitons},  Nonlinear Anal., 132(1), 66–94, (2016).



\bibitem{Chen} Chen, X., Lu, P., Tian, G.,  {\it A note on uniformization of Riemannian surfaces by Ricci flow},   Proc. Amer. Math. Soc., 134, 3391–3393, (2006).


\bibitem{bib16} Dwivedi, S.,  {\it Some results on Ricci-Bourguignon and  almost solitons},  Can. Math. Bull., 64(3), 591-604, (2021). 


\bibitem{Gom}  Gomes J. N,  Wang, Q.,  Xia, C., {\it On the $h$-almost Ricci soliton},  J. Geom. Phys., 144, 216–222, (2017). 


\bibitem{Haj} Ghahremani-Gol, H.,  {\it Some results on $h$-almost Ricci solitons},  J. Geom. Phys., 137, 212–216, (2019).


\bibitem{bib17} Hamilton, R.S., {\it Three-manifolds with positive Ricci curvature},  J. Differ. Geom., 17(2), 255-306, (1982). 

\bibitem{bib18}  Hamilton, R.S.,  {\it The Ricci flow on surfaces, Mathematics and general relativity}, Contemp. Math., 71,  237-262, (1988). 


\bibitem{bib19}  Ivey, T.,  {\it Ricci solitons on compact three-manifolds},  Differ. Geom. Appl., 3(4), 301–307, (1993). 

\bibitem{kazj}  Kazdan, J., Warner, F., {\it Existence and conformal deformation of metrics with prescribed Gaussian and scalar curvature}., Ann. Math. 101, 317–331 (1975).

\bibitem{bib20}  Lott, J.,  {\it On the long time behavior of type-III Ricci flow solutions},  Math. Ann., 339(3), 627-666, (2007). 
 

\bibitem{bib22}  Ni, L., Wallach, N., {\it On a classification of gradient shrinking solitons},  Math. Res. Lett., 15(5), 941–955, (2008). 

\bibitem{bib23}  Perelman, G.,  {\it The Entropy formula for the Ricci flow and its Geometric Applications}, arXiv preprint math/0211159 (2002)   

\bibitem{bib24}  Petersen, P., Wylie, W.,  {\it Rigidity of gradient Ricci solitons},  Pacific J. Math., 241(2), 329-345, (2009). 


\bibitem{Mk}  Traore, M., Ta\c{s}tan, H.M., Gerdan, G. A.,  \it{Some Characterizations on Gradient Almost $\eta$-Ricci-Bourguignon Solitons},  Bol. Soc. Paran. Mat. \textbf{43}, 1-12, (2025)

\bibitem{Bic}  Traore, M., Ndiaye, A., Sane, M.,  \it{Some Results on Warped Product  $\eta$-Ricci-Bourguignon Solitons}, Nonlinear Analysis, Geometry and Applications, Cham: Birkhäuser. Trends Math.. 299-316, (2024). 

\bibitem{bib30}  Traore, M., Ta\c{s}tan, H.M., Gerdan Ayd{\i}n, S.,  {\it On almost $\eta$-Ricci-Bourguignon solitons},  Miskolc. Math. Notes., 25(1), 493–508, (2024).

\bibitem{Filo}  Traore, M., Ta\c{s}tan, H.M., {\it  On sequential warped product $\eta$-Ricci-Bourguignon solitons},   Filomat, 38(19), 6785–6797, (2024).

\end{thebibliography}
\end{document}